\documentclass{amsart}
\title{Lie elements and the matrix-tree theorem}
\thanks{The research of
    the first named author was funded by the Russian Academic
    Excellence Project `5-100' and by the Simons--IUM fellowship 2020
    by the Simons Foundation. The second named author is supported by
    the Scholarship of the President of Russian Federation (2020).}
\author[Yu.Burman]{Yurii Burman}
\email{burman@mccme.ru}
\address{ National Research University Higher School of Economics, 119048, 6 
 Usacheva str., Moscow, Russia, and Independent University of Moscow, 
 119002, 11 B.Vlassievsky per., Moscow, Russia.}
\author[V.Kulishov]{Valeriy Kulishov}
\email{valerykulishov@mail.ru}
\address{National Research University Higher School of Economics,
  Moscow, Russia.}
\subjclass[2010]{05C50}
\keywords{Matrix-tree theorem, Lie algebra}

\date{}

\makeatletter

\RequirePackage{amssymb}

\newcommand{\theoremName}{Theorem}
\newcommand{\lemmaName}{Lemma}
\newcommand{\corollaryName}{Corollary}
\newcommand{\statementName}{Proposition}

\newcommand{\remarkName}{Remark}
\newcommand{\exampleName}{Example}
\newcommand{\definitionName}{Definition}
\newcommand{\problemName}{Problem}
\newcommand{\proofName}{Proof}
\renewcommand{\proofname}{\proofName}
\newcommand{\answerName}{Answer}
\newcommand{\hintName}{Hint}

\theoremstyle{plain}

\newtheorem {theorem}{\theoremName}
\newtheorem {lemma}{\lemmaName}
\newtheorem {corollary}{\corollaryName}

\newtheorem {proposition}{\statementName}

\theoremstyle{remark}

\newtheorem{remark}{\remarkName}
\newtheorem{Remark}{\remarkName}

\newtheorem{example}{\exampleName}

\theoremstyle{definition}

\newtheorem{definition}{\definitionName}

\let\@newpf\proof 
\let\proof\relax

\def \namepf[#1] {\@newpf[\proofname\ #1]}
\newenvironment{proof}{\@ifnextchar[{\namepf}{\@newpf[\proofname]}}{\qed\endtrivlist}

\newcounter{qst}
\@addtoreset{qst}{problem}

\def \Integer {{\mathbb Z}}

\def \Complex {{\mathbb C}}

\def \lnorm#1\rnorm {\vphantom{#1}\left\|\smash{#1}\right\|}
\def \lmod#1\rmod {\vphantom{#1}\left|\smash{#1}\right|}

\newcommand \bydef {\stackrel{\mbox{\scriptsize def}}{=}}

\@ifundefined{name}{\newcommand \name[1] {\mathop{\rm #1}\nolimits}}%
{\relax}

\renewcommand \phi {\varphi}
\renewcommand \rho {\varrho}

\newcommand{\DT}[1]{#1 \dots #1}

\def \sdet {\name{sdet}}
\def \sgn {\name{sgn}}
\def \Lie {\mathop{\mathcal L}\nolimits}
\def \Tree {\mathop{\mathcal T}\nolimits}
\def \I {\mathop{\mathcal I}\nolimits}
\def \G {\mathop{\Gamma}\nolimits}
\def \E {\mathop{\mathcal E}\nolimits}
\newcommand{\lr}[1]{\langle #1\rangle}
\def \Grp {\mathop{\mathcal G}\nolimits}
\def \Alg {\mathop{\mathcal A}\nolimits}
\def \kappa {\varkappa}

\theoremstyle{plain}
\newtheorem{conjecture}[theorem]{Conjecture}

\makeatletter
\@addtoreset{theorem}{section}
\let \c@lemma=\c@theorem
\let \c@proposition=\c@theorem
\let \c@definition=\c@theorem
\let \c@corollary=\c@theorem
\let \c@remark=\c@theorem
\let \c@example=\c@theorem
\makeatother

\def \symmdiff {\mathrel{\triangle}}

\begin{document}

\begin{abstract}
  For a finite-dimensional representation $V$ of a group $G$ we
  introduce and study the notion of a Lie element in the group algebra
  $k[G]$. The set $\Lie(V) \subset k[G]$ of Lie elements is a Lie
  algebra and a $G$-module acting on the original representation $V$.

  Lie elements often exhibit nice combinatorial properties. Thus, for
  a $G = S_n$ and $V$, a permutation representation, we prove a
  formula for the characteristic polynomial of a Lie element similar
  to the classical matrix-tree theorem.
\end{abstract}

\maketitle
 
\section{Introduction: Lie elements in the group algebra}

Let $V$ be a finite-dimensional representation of a group $G$ over a field 
$k$. For every $g \in G$ and every $m$ define linear operators $\Grp_m\lr{g}, 
\Alg_m\lr{g}: V^{\wedge m} \to V^{\wedge m}$ as follows: 
 \begin{align*}
\Grp_m\lr{g} (v_1 \wedge \dots \wedge v_m) &= g(v_1) \wedge \dots 
\wedge g(v_m)\\
\Alg_m\lr{g} (v_1 \wedge \dots \wedge v_m) &= \sum_{p=1}^m v_1
\wedge \dots \wedge g(v_p) \wedge \dots \wedge v_m.
 \end{align*}
(here and below $v_1, \dots, v_m$ are arbitrary vectors in $V$). Also
take by definition
\begin{align*}
  \Grp_0\lr{g} &= \I\\
  \Alg_0\lr{g} &= 0
\end{align*}
for every $g \in G$. Here and below $\I$ means the identity operator.

Denote by $k[G]$ the group algebra of $G$; extend $\Grp_m$ and
$\Alg_m$ by linearity to operators $k[G] \to \name{End}(V^{\wedge
  m})$. In particular, $\Alg_0 = 0$ and $\Grp_0\lr{\sum_{g \in G} a_g
  g} = \sum_{g \in G} a_g$ (a constant regarded as an operator $k \to
k$).

 \begin{definition}
An element $x \in k[G]$ satisfying $\Grp_m\lr{x} = \Alg_m\lr{x}$ for all $m = 0,1, 
\dots, \dim V$ is called a Lie element (with respect to the representation 
$V$). The set of Lie elements is denoted by $\Lie(V) \subset k[G]$.
 \end{definition}

 \begin{remark}\label{Rm:Sum0}
In particular, if $x = \sum_{g \in G} a_g g$ is a Lie element then $\sum_{g 
\in G} a_g = \Grp_0\lr{x} = \Alg_0\lr{x} = 0$.
 \end{remark}

 \begin{example}
Let $G = S_n$ (a permutation group), $k = \Complex$ and $V =
\Complex^n$, the permutation representation of $S_n$ (an element of
the group permutes the coordinates of a vector $v = (x_1 \DT, x_n) \in
\Complex^n$).

\begin{lemma}\label{Lm:KirchDiff}
  $\kappa_{ij} \bydef 1 - (ij) \in \Complex[S_n]$ is a Lie element.
\end{lemma}

Here $(ij) \in S_n$ means a transposition of $i$ and $j$; more
generally, we will use notation like $(i_1 \dots i_k)$ for a cyclic
element in $S_n$, that is, a permutation sending $i_1 \mapsto i_2
\DT\mapsto i_k \mapsto i_1$ and leaving all other elements of $\{1
\DT, n\}$ intact.

We call $\kappa_{ij} \in \Complex[S_n]$ a {\em Kirchhoff difference}
as a tribute to G.\,Kirchhoff's seminal paper \cite{Kirchhoff} (1847);
see Theorem \ref{Th:MTT} below.

\begin{proof}[of Lemma \ref{Lm:KirchDiff}]
  The proof is a direct computation. First, $\Grp_m\lr{1} = \I$ and $\Alg_m\lr{1}
  = m\I$. Obviously (cf.\ Proposition \ref{Pp:Conj} below), one can assume
  $i=1$, $j=2$ without loss of generality.  Denote by $v_1 \DT, v_n$ the
  standard basis in $\Complex^n$; by linearity, it is enough to
  consider the action of $\Grp_m\lr{(12)}$ and $\Alg_m\lr{(12)}$ on $v = v_{i_1}
  \DT\wedge v_{i_m}$ where $1 \le i_1 \DT< i_m \le n$.

  Consider now three cases:

  \begin{itemize}
  \item $i_1 \ge 3$: here $\Grp_m\lr{(12)}v = v$ and $\Alg_m\lr{(12)}v = mv$, and
    therefore $\Grp_m\lr{1-(12)}v = 0 = \Alg_m\lr{1-(12)}v$.

  \item $i_1 = 2$: here $\Grp_m\lr{(12)}v = v_1 \wedge v_{i_2}
    \DT\wedge v_{i_m}$ and $\Alg_m\lr{(12)} v = (v_1 + (m-1)v_2)
    \wedge v_{i_2} \DT\wedge v_{i_m}$, so that $\Grp_m\lr{1-(12)}v =
    \Alg_m\lr{1-(12)}v = (v_2-v_1) \wedge v_{i_2} \DT\wedge
    v_{i_m}$. The case $i_1=1$ and $i_2 \ge 3$ is similar.

  \item finally, $i_1=1, i_2=2$: here $\Grp_m\lr{(12)}v = -v$, so
    $\Alg_m\lr{(12)} v = (m-2) v$, and therefore $\Grp_m\lr{1-(12)}v =
    2v = \Alg_m\lr{1-(12)}v$.
  \end{itemize}
  Lemma is proved.
 \end{proof}
 \end{example}

Our first motivation to study Lie elements was the paper
\cite{BurmanZvonkine} where the Lie element property of the Kirchhoff
differences was used to study a question in low-dimensional topology
(see \cite[Proposition 3.4]{BurmanZvonkine}). Another reason that
makes Lie elements interesting are nice combinatorial properties of
the elements $x \in \Lie(\Complex^n)$; see e.g.\ the classical
matrix-tree theorem (Theorem \ref{Th:MTT}) and its Pfaffian version by
G.\,Masbaum and A.\.Vaintrob (Theorem \ref{Th:PfT}) below. The main
result of this paper, Theorem \ref{Th:Main}, is an analog of Theorems
\ref{Th:MTT} and \ref{Th:PfT}.

For any $G$ and $V$ the $k$-vector spaces $k[G]$ and
$\name{End}(V^{\wedge m})$ are associative algebras; consider them as
Lie algebras with the commutator bracket: $[p,q] \bydef pq-qp$.

 \begin{proposition}\label{Pp:Lie} 
Maps $\Grp_m, \Alg_m: k[G] \to \name{End}(V^{\wedge m})$ are Lie algebra 
homomorphisms. 
 \end{proposition}

 \begin{proof}
Obviously, $\Grp_m: k[G] \to \name{End}(V^{\wedge m})$ is an associative 
algebra homomorphism, hence a Lie algebra homomorphism. For $\Alg_m$ take $x = 
\sum_{g \in G} a_g g$, $y = \sum_{h \in G} b_h h$, to obtain
 \begin{align*}
\Alg_m\lr{x} &\Alg_m\lr{y} v_1 \wedge \dots \wedge v_m = \sum_{h \in 
G, 1 \le p \le m} b_h \Alg_m\lr{x} v_1 \wedge \dots \wedge 
h(v_p) \wedge \dots \wedge v_m\\
&= \sum_{g,h \in G, 1 \le p \le m} a_g b_h v_1 \wedge \dots \wedge 
g(h(v_p)) \wedge \dots \wedge v_m \\
&\hphantom{=}+ \sum_{g,h \in G, 1 \le p,q \le m, p \ne 
q} a_g b_h v_1 \wedge \dots \wedge h(v_p) \wedge \dots \wedge g(v_q) \wedge 
\dots \wedge v_m,
 \end{align*}
implying $\Alg_m\lr{[x,y]} = [\Alg_m\lr{x}, \Alg_m\lr{y}]$.
 \end{proof}

 \begin{corollary}\label{Cr:LieAlg}
The set of Lie elements $\Lie(V) \subset k[G]$ is a Lie subalgebra.
 \end{corollary}

 \begin{proposition} \label{Pp:Conj}
Operators $\Grp_m$ and $\Alg_m$ are conjugation-invariant: if $x \in k[G]$ and $y 
\in k[G]$ is invertible then for any $m$ one has $y \Grp_m\lr{x} y^{-1} = 
\Grp_m\lr{yxy^{-1}}$ and $y \Alg_m\lr{x} y^{-1} = \Alg_m\lr{yxy^{-1}}$.
 \end{proposition}

The proof is straightforward.

 \begin{corollary}
The Lie algebra $\Lie(V) \subset k[G]$ is a representation of $G$
where elements of the group act by conjugation.
 \end{corollary}

\section{Lie elements in the relection representation of the permutation group}

\subsection{Kirchhoff differences and their commutators}

Let $G = S_n$ (a permutation group), and $V$ be its reflection
(a.k.a.\ Coxeter or geometric) representation; $\dim V = n-1$.  The
permutation representation $\Complex^n$ is a sum of $V = \{(x_1 \DT,
x_n) \mid x_1 \DT+ x_n = 0\}$ and a trivial representation $\mathbf 1
= \{(x, x \DT, x) \mid x \in \Complex\}$. It follows from Remark
\ref{Rm:Sum0} that any Lie element $x \in \Lie(\Complex^n)$ acts on
$\mathbf 1$ by zero. Therefore, $\Lie(V) = \Lie(\Complex^n)$; we'll
denote it $\Lie_n$ for short.

For a finite-dimensional representation $W$ of $S_n$ denote by
$\chi_x^W(t) = \det (t\I-x)$ the characteristic polynomial of an
element $x \in \Complex[S_n]$ acting in $W$. It follows from the
remarks above that $\chi_x^{\Complex^n}(t) = t\chi_x^V(t)$ for all $x
\in \Lie_n$.

 \begin{theorem}[\cite{LiePrepr}, cf.\ \cite{BurmanZvonkine}]\label{Th:Transp}
   \begin{enumerate}
\item\label{It:Lie} For all pairwise distinct $i, j, k, l \in \{1 \DT,
  n\}$ elements $\kappa_{ij} \bydef 1 - (ij)$, $\nu_{ijk} \bydef
  (ijk)-(ikj)$ and $\eta_{ijkl} \bydef (ijkl)+(ilkj)-(ijlk)-(iklj)$
  belong to $\Lie_n$.

\item\label{It:Dim} Consider, for all $1 \le i < j < k < l \le n$,
  vector spaces $K_{ij}, N_{ijk}, H_{ijkl} \subset \Complex[S_n]$
  spanned by all $\kappa_{pq}$, $\nu_{pqr}$, $\eta_{pqrs}$ where the
  indices $(p,q)$, $(p,q,r)$ and $(p,q,r,s)$ are permutations of
  $(i,j)$, $(i,j,k)$ and $(i,j,k,l)$, respectively. Then $\dim K_{ij}
  = 1$, $\dim N_{ijk} = 1$ and $\dim H_{ijkl} = 2$; bases in them
  are $\{\kappa_{ij}\}$, $\{\nu_{ijk}\}$ and $\{\eta_{ijkl},
  \eta_{iklj}\}$.

\item\label{It:Repr} Let permutation groups $S_2$, $S_3$ and $S_4$ act
  on $K_{ij}$, $N_{ijk}$ and $H_{ijkl}$ permuting indices of the
  elements $\kappa_{pq}$, $\nu_{pqr}$ and $\eta_{pqrs}$. This makes
  $K_{ij}$ a trivial representation of $S_2$, $N_{ijk}$, a sign
  representation of $S_3$, and $H_{ijkl}$, an irreducible
  $2$-dimensional representation of $S_4$.

\item\label{It:Relat} Elements $\kappa_{pq} \in K_{ij}$, $\nu_{pqr}
  \in N_{ijk}$ and $\eta_{pqrs} \in H_{ijkl}$ enjoy the following
  symmetries (for all $p, q, r, s$):\\
  for $K_{ij}$:
  \begin{equation*}
    \kappa_{qp} = \kappa_{pq};
  \end{equation*}
  for $N_{ijk}$:
  \begin{equation*}
    \nu_{pqr} = \nu_{qrp} = -\nu_{qpr};    
  \end{equation*}
  for $H_{ijkl}$:
  \begin{align*}
    &\eta_{pqrs} = -\eta_{qprs} = -\eta_{pqsr}, \\
    &\eta_{pqrs} = \eta_{srqp} = \eta_{rspq} = \eta_{qpsr},\\
    &\eta_{pqrs} + \eta_{prsq} + \eta_{psqr} = 0.
  \end{align*}
  \end{enumerate}
 \end{theorem}

 \begin{proof}
  Kirchhoff differences $\kappa_{ij}$ belong to $\Lie_n$ by Lemma
  \ref{Lm:KirchDiff}. $\nu_{ijk}$ and $\eta_{ijkl}$ are commutators of
  the $\kappa_{ij}$: $\nu_{ijk} = [\kappa_{ij},\kappa_{jk}]$ and
  $\eta_{ijkl} = [\kappa_{il},\nu_{ijk}]$. So by Corollary
  \ref{Cr:LieAlg} assertion \ref{It:Lie} is proved.

  Relations of assertion \ref{It:Relat} can be checked immediately.  A
  straightforward computation shows that these relations imply
  assertions \ref{It:Dim} and \ref{It:Repr}. The first relation for
  $H_{ijkl}$ means that the basic elements $\eta_{ijkl}, \eta_{iklj}$
  are eigenvectors of the transpositions $(12)$ and $(13)$,
  respectively, with the eigenvalue $-1$.
 \end{proof}

 \begin{conjecture}\label{Cj:KirchGen}
The Lie algebra $\Lie_n$ is generated by the Kirchhoff differences
$\kappa_{ij}$, $1 \le i < j \le n$.
 \end{conjecture}

This conjecture was tested numerically for small $n$, but we do not
know its proof at the moment.
 
Characterstic polynomials of Lie elements $x \in \Lie_n$ acting at $V$
are often given by nice formulas.

 \begin{example}
Let $\Gamma$ be a finite graph with the vertex set $\{1 \DT, n\}$ and
the edges $e_1 \DT, e_m$ where $e_s$ connects vertices $i_s$ and
$j_s$; denote $w_\Gamma \bydef w_{i_1 j_1} \dots w_{i_m j_m}$. Also
denote by $\Tree_n$ the set of trees with the vertices $1 \DT, n$.

Consider the Lie element
 \begin{equation*}
x = \sum_{1 \le i < j \le n} w_{ij} \kappa_{ij};
 \end{equation*}
and assume $w_{ji} = w_{ij}$ for convenience.

 \begin{theorem}[matrix-tree theorem, \cite{Kirchhoff}] \label{Th:MTT}
 \begin{equation*}
\det \left. x\right|_V = \chi_x^V(0) = n\sum_{\Gamma \in \Tree_n} w_\Gamma.
 \end{equation*}
 \end{theorem}

There exist similar formulas for other coefficients of $\chi_x^V$ as
well; for details see the review \cite{Chaiken} and the references
therein.
 \end{example}

 \begin{example}
A finite $3$-graph is defined as a union of several solid triangles
(called $3$-edges) with some of their vertices glued. A $3$-graph is
called a $3$-tree if it is contractible (as a topological space). The
number $n$ of vertices of a $3$-tree is always odd: $n = 2m+1$ where
$m$ is the number of $3$-edges; denote by $\Tree_m^{(3)}$ the set of
$3$-trees with the vertices $1 \DT, 2m+1$.

Consider the Lie element
\begin{equation*}
y = \sum_{1 \le  i < j < k \le n} w_{ijk}\nu_{ijk};
\end{equation*}
assume for convenience $w_{jki} = w_{kij} = w_{ijk}$ and $w_{jik} =
w_{ikj} = w_{kji} = -w_{ijk}$ (cf.\ assertion \ref{It:Relat} of
Theorem \ref{Th:Transp}).

Let $\Gamma$ be a $3$-graph and $e_1 \DT, e_m$, its $3$-edges; the
edge $e_s$ is a triangle with the vertices $i_s, j_s, k_s \in \{1 \DT,
n\}$. Denote $w_\Gamma \bydef w_{i_1 j_1 k_1} \dots w_{i_m j_m k_m}$.

It is easy to observe that the operator $\nu_{ijk}: V \to V$ (and
hence, the operator $y: V \to V$) is skew-symmetric with respect to
the standard scalar product in $V$ (inherited from $\Complex^n$). So
if $n$ is even and $\dim V = n-1$ is odd, then $\det \left. y\right|_V =
0$. If $n = 2m+1$ is odd then the skew-symmetric operator
$\left. y\right|_V$ has a Pfaffian described below.

Folllowing \cite{MV}, define a sign $\delta(\Gamma) = \pm1$ of a
$3$-tree $\Gamma \in \Tree_m^{(3)}$ as follows. Denote, like above,
the vertices of the $s$-th edge $e_s$ of $\Gamma$ as $i_s < j_s <
k_s$; here $s = 1 \DT, m$. Consider a product of the $3$-cycles
$\sigma \bydef (i_1 j_1 k_1) \dots (i_m j_m k_m) \in S_n$. An easy
induction by $m$ shows that $\sigma$ is a cyclic permutation $(a_1
\dots a_n)$. Now define a permutation $\tau \in S_n$ as $\tau(s) =
a_s$, $s = 1 \DT, n$; the sign $\delta(\Gamma)$ is then defined as the
parity of $\tau$. See \cite{MV} for details; in particular, it is
proved there that $\delta(\Gamma)$ does not depend on the ordering of
the edges of $\Gamma$.

 \begin{theorem}[\cite{MV}] \label{Th:PfT}
$\name{Pf} \left.y\right|_V = n \sum_{\Gamma \in \Tree_n^{(3)}}
   \delta(\Gamma) w_\Gamma$.
 \end{theorem}
 \end{example}

The article \cite{BPT} describes a technique (called discrete path
integration) giving a uniform proof of Theorems \ref{Th:MTT} and
\ref{Th:PfT} and of some more similar statements as well. We are going
to use this technique to prove the main result of the paper, Theorem
\ref{Th:Main}.

\subsection{The main theorem}\label{Sec:MainTh}

Theorem \ref{Th:Main} is a formula for the characteristic polynomial
of the Lie element
 \begin{equation}\label{Eq:WSum}
z = \sum_{1 \le i < j < k < l \le n} \xi_{ijkl}: V \to V.
 \end{equation}
where $\xi_{ijkl} \in H_{ijkl}$ are arbitrary elements (see Theorem
\ref{Th:Transp} above), that is, $\xi_{ijkl} = w_{ijkl} \eta_{ijkl} +
w_{iklj} \eta_{iklj}$ for some $w_{ijkl}, w_{iklj} \in \Complex$.
 
Let $A = (a_{ij})$ and $B = (b_{ij})$ be $n \times n$-matrices, and $I
\subseteq \{1 \DT, n\}$. Their {\em $I$-shuffle} is defined as a $n
\times n$-matrix $(A,B)_I = (u_{ij})$ where
 \begin{equation*}
u_{ij} = \begin{cases}
a_{ij}, &i \in I,\\
b_{ij}, &i \notin I.
 \end{cases}
 \end{equation*}

 \begin{definition}\label{Df:SDet}
The {\em shuffle determinant} of the matrices $A$ and $B$ is
 \begin{equation*}
\sdet(A,B) \bydef \sum_{I \subseteq \{1 \DT, n\}} \det (A,B)_I \det 
(A,B)_{\bar I}.
 \end{equation*}
where bar means the complement: $\bar I \bydef \{1 \DT, n\} \setminus I$. 
 \end{definition}

Let $\E$ be a $r$-element set, which elements are $4$-tuples
$(i_1,j_1,k_1,l_1) \DT, (i_r,j_r,k_r,l_r)$; here $1 \le r \le
n$. Consider the vector space $H_{\E} = \bigotimes_{s=1}^r
H_{i_sj_sk_sl_s}$ of dimension $2^r$ and define a linear functional
$\Phi_r: H_{\E} \to \Complex$ as follows. Let $A_{\E}, B_{\E}$ be $r
\times n$-matrices with the elements
\begin{equation}\label{Eq:DefAB}
  \begin{aligned}
    &(A_{\E})_{si_s} = 1, (A_{\E})_{sj_s} = -1, \\
    &(B_{\E})_{sk_s} = 1, (B_{\E})_{sl_s} = -1,\\
    &(A_{\E})_{ij} = (B_{\E})_{ij} = 0 \qquad\text{for all other $i,j$};
  \end{aligned}
\end{equation}
here $s = 1 \DT, r$. For a $r$-element set $J \subseteq \{1 \DT, n\}$
denote by $A_{\E}^J$ and $B_{\E}^J$ the $r \times r$-submatrices of
$A_{\E}$ and $B_{\E}$, respectively, containing all the $r$ rows
and the columns listed in $J$. Then take by definition

\begin{equation}\label{Eq:DefPhi}
  \Phi_r(\eta_{i_1 j_1 k_1 l_1} \DT\otimes \eta_{i_r j_r k_r l_r}) =
  \sum_{J \subseteq \{1 \DT, n\}, \#J=r} \sdet(A_{\E}^J,B_{\E}^J).
\end{equation}
and extend $\Phi_r$ to the whole $H_{\E}$ by linearity.

\begin{theorem}\label{Th:Main}
Let $z$ be defined by \eqref{Eq:WSum} and $\Phi_r$, by
\eqref{Eq:DefPhi}. Then 
  \begin{equation*}
    \chi_z^{\Complex^n}(t) = t^n + \mu_1 t^{n-1} \DT+ \mu_{n-1}t,
  \end{equation*}
where
  \begin{equation*}
\mu_r = \Phi_r(z^{\otimes r}).
  \end{equation*}
for every $r = 1 \DT, n-1$.  
\end{theorem}

See Section \ref{Sec:ProofMain} for the proof.

\begin{Remark}
If $r = n-1$ then the definition of $\Phi_r$ can be simplified:

\begin{proposition}\label{Pp:Phin-1}
  If $r = n-1$ then all the summands in \eqref{Eq:DefPhi} are equal,
  so one may take
  \begin{equation*}
\Phi_{n-1} (\eta_{i_1 j_1 k_1 l_1} \DT\otimes \eta_{i_{n-1} j_{n-1} k_{n-1} l_{n-1}}) =
  n \sdet(A_{\E}^J,B_{\E}^J).
  \end{equation*}
where $J$ is any subset of $\{1 \DT, n\}$ of cardinality $(n-1)$,
e.g.\ $J = \{1 \DT, n-1\}$.
\end{proposition}

The proof of the proposition is also in Section \ref{Sec:ProofMain}.
\end{Remark}

\section{Shuffle determinant}

Here are basic properties of the shuffle determinant of Definition
\ref{Df:SDet}:

 \begin{theorem}\label{Th:SDetBasic}
 \begin{enumerate}
\item\label{It:Terms} $\sdet(A,B)$ is a polynomial of variables
  $a_{ij}$ and $b_{ij}$, $1 \le i, j \le n$, with integer coefficients,
  bihomogeneous of degree $n$ (thus, its total degree is $2n$).

\item\label{It:Symm} $\sdet(B,A) = \sdet(A,B)$. 

\item\label{It:Coeff} Let $X \bydef \name{diag}(x_1 \DT, x_n)$ be a
  diagonal matrix with $x_1 \DT, x_n$ as diagonal entries. Then
  $\sdet(A,B) = [x_1 \dots x_n:\det(A + BX)^2]$ (that is, $\sdet(A,B)$
  is equal to the coefficient at the monomial $x_1 \dots x_n$ in the
  polynomial $\det(A + BX)^2$).

\item\label{It:LMult} $\sdet(CA, CB) = \sdet(A,B) \det C^2$ for any $n 
\times n$-matrix $C$. In particular, if $B$ is invertible then $\sdet(A,B) = 
\sdet(B^{-1}A, \I) \det B^2 $.

\item\label{It:E} $\sdet(A,\I) = (-1)^n \sum_{\sigma \in S_n}
  (-2)^{\nu(\sigma)} a_{1\sigma(1)} \dots a_{n\sigma(n)}$ where
  $\nu(\sigma)$ the number of independent cycles in $\sigma$.
 \end{enumerate}
 \end{theorem}

 \begin{proof}
Assertions \ref{It:Terms} and \ref{It:Symm} are obvious from  
Definition \ref{Df:SDet}.

Assertion \ref{It:Coeff}: denote by $a_1, \DT, a_n$ and $b_1, \DT, b_n$ 
columns of the matrices $A$ and $B$, respectively; we will be writing 
$\det(a_1 \DT, a_n)$ instead of $\det A$, and similarly for other 
matrices. The determinant of a matrix is a multilinear function of its 
columns, so one has 
 \begin{align*}
&[x_1 \dots x_n : \det(A+BX)^2] = [x_1 \dots x_n : \det(a_1+x_1b_1 \DT, 
a_n+x_nb_n)^2] \\
&= \sum_{I \subseteq \{1 \DT, n\}} [x_I: \det(a_1+x_1b_1 \DT, a_n+x_nb_n)] 
[x_{\bar I}: \det(a_1+x_1b_1 \DT, a_n+x_nb_n)] \\
&\text{(where $x_I \bydef \prod_{i \in I} x_i$)}\\
&\hphantom{[x_1 \dots x_n : \det(A+BX)^2]} = \sum_{I \subseteq \{1
  \DT, n\}} \det(w_1 \DT, w_n) \det(w'_1 \DT, w'_n) \\
&\text{(where $w_i = b_i$, $w'_i = a_i$ if $i\in I$ and vice versa if $i 
\notin I$)}\\
&\hphantom{[x_1 \dots x_n : \det(A+BX)^2]} = \sum_{I \subseteq \{1
  \DT, n\}} \det(A,B)_I \det(A,B)_{\bar I} = \sdet(A,B).
 \end{align*}

Assertion \ref{It:LMult} follows from \ref{It:Coeff}: $\det (CA +
CBX)^2 = \det (A + BX)^2 \det C^2$. The matrix $C$ does not depend on
$x_1 \DT, x_n$, so the same equality takes place for coefficients at
$x_1 \dots x_n$.

To prove assertion \ref{It:E} note that $\det (A,\I)_I$ is the
diagonal minor of the matrix $A$ comprising the rows and the columns
with the numbers in $I$. Hence,
 \begin{equation*}
\det (A,\I)_I = \sum_{\substack{\sigma \in S_n\\ \sigma(j) = j\, \forall j 
\in \bar I}} \sgn(\sigma) \prod_{i \in I} a_{i \sigma(i)},
 \end{equation*}
where $\sgn\sigma = 1$ or $-1$ depending on the parity of $\sigma$. 
Therefore 
 \begin{equation*}
\det (A,\I)_I \det (A,\I)_{\bar I} = \sum_{\substack{\sigma \in S_n\\
\sigma(i) \in I\, \forall i \in I\\
\sigma(i) \in \bar I\, \forall i \in \bar I}}
\sgn(\sigma) a_{1\sigma(1)} \dots a_{n\sigma(n)}.
 \end{equation*}
Summation over $I \subseteq \{1 \DT, n\}$ gives
 \begin{align*}
\sdet(A,\I) &= \sum_I \det (A,\I)_I \det (A,\I)_{\bar I} \\
&= \sum_{\sigma \in S_n} \#\{I \subseteq \{1 \DT, n\} \mid \sigma(i) \in 
I\, \forall i \in I\} \sgn(\sigma) a_{1\sigma(1)} \dots a_{n\sigma(n)}.
 \end{align*}
The subset $I$ invariant with respect to $\sigma$ (that is, such that
$\sigma(i) \in I$ for all $i \in I$) is a union of several independent
cycles of $\sigma$; so the number of such subsets is
$2^{\nu(\sigma)}$. On the other hand, $\sgn(\sigma) =
(-1)^{n+\nu(\sigma)}$, which finishes the proof.
 \end{proof}

Give now a more detailed description of $\sdet(A,B)$ as a polynomial
of $a_{ij}$ and $b_{ij}$. By assertion \ref{It:Terms} of Theorem
\ref{Th:SDetBasic} any term of the polynomial looks like ${c \cdot
  a_{i_1 j_1} \dots a_{i_n j_n} b_{k_1 l_1} \dots b_{k_n l_n}}$ where
$c \in \Integer$. Denote by $\G \bydef \G(i_1, j_1 \DT, k_n, l_n)$ a
directed graph with the vertices $1 \DT, n$ and the $2n$ edges $(i_1
j_1) \DT, (i_n j_n),\linebreak[1] (k_1 l_1) \DT, (k_n l_n)$.

 \begin{theorem}
 \begin{enumerate}
\item\label{It:2in2out} Every vertex of the graph $\G$ is incident to
  exactly four edges; the vertex is initial for two of them and is
  terminal for the remaining two.

\item\label{It:ABSymm} The coefficient $c$ at the monomial depends on
  the graph $\G$ only and is equal to $\pm 2^{m(\G)}$ where $m(\G) \in
  \Integer_{>0}$ is the number of connected components of an auxiliary
  graph $\Gamma'$ determined by $\G$.
 \end{enumerate}
 \end{theorem}

The proof below contains the exact contruction of the graph $\Gamma'$. 

 \begin{proof}
Assertion \ref{It:2in2out}: take some $I \subseteq \{1 \DT, n\}$. If
$i \in I$ then the elements $a_{ij}$ (for all $j$) are in the $i$-th
column of $(A,B)_I$; if $i \notin I$, then they are in the $i$-th
column of $(A,B)_{\bar I}$. Thus, exactly one of $a_{i_1 j_1} \DT,
a_{i_n j_n}$ is $a_{ij}$ for some $j$, which implies $\{i_1 \DT, i_n\}
= \{1 \DT, n\}$; similarly, $\{k_1 \DT, k_n\} = \{1 \DT, n\}$. So,
every vertex of $\G$ is an initial vertex of two edges. At the same
time, for every $j \in \{1 \DT, n\}$ every monomial of $\det (A,B)_I$
contains exactly one letter $x_{ij}$ where $x = a$ or $b$, for some $i
\in \{1 \DT, n\}$; the same is true for $\det (A,B)_{\bar I}$ ---
hence, every vertex of $\G$ is a terminal vertex for two edges.

Assertion \ref{It:ABSymm}: note first that the monomial is {\em not}
determined uniquely by the graph $\G$ --- one cannot tell which edges
correspond to $a_{ij}$ and which to $b_{kl}$. Prove that this
ambiguity does not influence the coefficient.

By assertion \ref{It:2in2out}, every mononial in $\sdet(A,B)$ is equal
to $x = a_{1j_1}a_{2j_2} \dots a_{nj_n} \linebreak[1] b_{1l_1}b_{2l_2}
\dots b_{nl_n}$ for some $j_1 \DT, j_n, l_1 \DT, l_n$. It is enough to
show that the coefficient at $x$ is the same as the coefficient at the
monomial $x' = b_{1j_1}a_{2j_2} \dots a_{nj_n}a_{1l_1}b_{2l_2} \dots
b_{nl_n}$. Note that for every $I \subseteq \{1 \DT, n\}$ the
contribution of the term $\det (A,B)_I \det (A,B)_{\bar I}$ to the
coefficient at $x$ is equal to the contribution of $\det (A,B)_{I'}
\det (A,B)_{\bar I'}$ to the coefficient at $x'$, where $I' \bydef I
\symmdiff \{1\}$. But $I \mapsto I \symmdiff \{1\}$ is an invertible
operation (indeed, an involution) on the set of subsets of ${\{1 \DT,
  n\}}$, so the coefficients at $x$ and at $x'$ are equal.

To obtain a formula for the coefficient take a monomial $x$ as above
and paint every edge $(ij)$ of the graph $\G$ blue if the
corresponding letter comes from $I$ (that is, $i \in I$ and the letter
is $b_{ij}$ or $i \in \bar I$ and the letter is $a_{ij}$) and red if
it comes from $\bar I$. A blue-red painting of the edges of $\G$
corresponds to a subset $I \subseteq \{1 \DT, n\}$ if each vertex is
initial and terminal for exactly one red and one blue edge; if $I$
exists, then it is obviously unique. The subgraphs of $\G$ formed by
red and blue edges are graphs of some permutations; call them
$\sigma_r$ and $\sigma_b$, respectively. The contribution of the term
$\det (A,B)_I \det (A,B)_{\bar I}$ into the coefficient is equal to
the product of parities of $\sigma_r$ and $\sigma_b$.

Consider a graph $\Gamma'$ whose vertices are edges of $\G$; two
vertices are connected by an edge if the corresponding edges of $\G$
share the same initial vertex or the same terminal vertex. By
assertion \ref{It:2in2out}, every vertex of $\Gamma'$ is incident to
exactly two edges --- hence, $\Gamma'$ is a union of nonintersecting
cycles. Red and blue vertices alternate in the cycle; therefore, each
cycle in $\Gamma'$ has even length.

The graph $\G = \G(i_1, j_1 \DT, k_n, l_n)$ determines $\Gamma'$. To
fix a subset $I$ one should paint vertices of $\Gamma'$ so that the
colors alternate in every cycle. For each cycle there are obviously
two such paintings possible; thus, the number of subsets $I$ for the
graph $G$ is $2^m$ where $m$ is the number of cycles (connected
components) in $\Gamma'$.

Let now $I_1, I_2 \subseteq \{1 \DT, n\}$ be two sets making nonzero 
contributions to the coefficient at the monomial $x$ and such that the 
corresponding colorings differ on one cycle of the graph $\Gamma'$ only; let 
this cycle be $e_1 \dots e_{2s}$. Then permutations $(\sigma_1)_r$ and 
$(\sigma_2)_r$ differ by a product of transpositions $(e_1 e_2) (e_3 e_4) 
\dots (e_{2s-1} e_{2s})$, and their parities differ by $(-1)^s$.
The same is true for permutations $(\sigma_1)_b$ and $(\sigma_2)_b$, so the 
terms $\det (A,B)_{I_1} \det (A,B)_{\bar I_1}$ and $\det (A,B)_{I_2} \det 
(A,B)_{\bar I_2}$ make equal contributions of $\pm1$ into the coefficient. 
This finishes the proof.
 \end{proof}

\section{Proof of Theorem \ref{Th:Main} and final remarks}\label{Sec:ProofMain}

\subsection{Proofs}

\begin{proof}[of Proposition \ref{Pp:Phin-1}]
  The set $J \subset \{1 \DT, n\}$ of cardinality $(n-1)$ is $\{1 \DT,
  n\} \setminus \{k\}$ for some $k$; denote $A^J \bydef A_k$ and $B^J
  \bydef B_k$ for short. By definition, $\sdet(A_k,B_k) = \sum_{I
    \subseteq \{1 \DT, n\}} \det (A_k,B_k)_I \det (A_k,B_k)_{\bar
    I}$. Denote by $\xi_1 \DT, \xi_n$ the columns of the $(n-1) \times
  n$-matrix $(A,B)_I$; one has $\xi_1 \DT+ \xi_n = 0$. The matrix
  $(A_{k+1},B_{k+1})_I$ is obtained from $(A_k,B_k)_I$ by replacement
  of the column $\xi_{k+1}$ with $-\xi_1 \DT- \xi_{k-1} - \xi_{k+1}
  \DT- \xi_n$. Since all the rows $\xi_i$ except $\xi_{k+1}$ are
  present in $(A_{k+1},B_{k+1})_I$, the determinant of
  $(A_{k+1},B_{k+1})_I$ is the same as if the replacement row were
  still $-\xi_{k+1}$. Thus, $\det (A_{k+1},B_{k+1})_I = -\det
  (A_k,B_k)_I$. The subset $I$ is arbitrary, so $\det
  (A_{k+1},B_{k+1})_{\bar I} = -\det (A_k,B_k)_{\bar I}$, too, and
  therefore $\sdet(A_{k+1},B_{k+1}) = \sdet(A_k,B_k)$, proving the
  proposition.
\end{proof}

\begin{proof}[of Theorem \ref{Th:Main}]

Let $v, \alpha \in \Complex^n$ be nonzero vectors. Denote by
$M[\alpha,v]: \Complex^n \to \Complex^n$ a rank $1$ linear operator
defined as $M[\alpha,v](u) = (\alpha,u) v$, $u \in \Complex^n$, where
$(\cdot,\cdot)$ is the standard ($\Complex$-valued) scalar product in
$\Complex^n$ .

 \begin{lemma}\label{Lm:Rank2}
Let $v_1 \DT, v_n$ be the standard basis in $\Complex^n$ (orthonormal
with respect to $(\cdot,\cdot)$). Then the Lie element $\eta_{ijkl} =
(ijkl)+(ilkj)-(ijlk)-(iklj)$ acts in the permutation representation
$\Complex^n$ as $M[v_i-v_j,v_l-v_k] + M[v_l-v_k,v_i-v_j]$.
 \end{lemma}

The proof is an immediate check.

Lemma \ref{Lm:Rank2} allows to derive Theorem \ref{Th:Main} from
\cite[Corollary 2.4]{BPT}. To keep up with the notation of \cite{BPT},
let's take by definition
\begin{equation}\label{Eq:DefEAlpha}
  \begin{aligned}
    e_{s,0} &= v_{i_s}-v_{j_s},\\
    e_{s,1} &= v_{k_s}-v_{l_s},\\
    \alpha_{s,0} &= v_{k_s}-v_{l_s},\\
    \alpha_{s,1} &= v_{i_s}-v_{j_s},    
  \end{aligned}
\end{equation}
so that $z = \sum_{s=1}^m \sum_{u \in \{0,1\}}
M[e_{s,u},\alpha_{s,u}]$. Now Corollary 2.4 of \cite{BPT} implies that
\begin{align*}
  \mu_r &= \sum_{s_1 \DT, s_r = 1}^m w_{i_{s_1}j_{s_1}k_{s_1}l_{s_1}}
  \dots w_{i_{s_r}j_{s_r}k_{s_r}l_{s_r}} \sum_{u_1 \DT, u_r = 0}^1
  \det \bigl((\alpha_{s_p,u_p}, e_{s_q,u_q})\bigr)_{p,q=1}^r \\
  &= \sum_{s_1 \DT, s_r = 1}^m w_{i_{s_1}j_{s_1}k_{s_1}l_{s_1}} \dots
  w_{i_{s_r}j_{s_r}k_{s_r}l_{s_r}} \\
  &\hphantom{= \sum_{s_1 \DT, s_r = 1}^m}\times
  \sum_{u_1 \DT, u_r = 0}^1 \sum_{\substack{J
    \subseteq \{1 \DT, n\}\\ \#J = r}} \det (\alpha_{s_p,u_p})_{p \in
    J} \det (e_{s_p,u_p})_{p \in J};
\end{align*}
in the last equation by $\det(c_1 \DT, c_r)$ we mean a determinant of
a $r \times r$ matrix having vectors $c_1 \DT, c_r \in \Complex^r$ as
columns. Instead of indices $u_1 \DT, u_r \in \{0,1\}$ consider a set
$I \bydef \{i \mid u_i = 1\} \subseteq \{1 \DT, r\}$. Taking
\eqref{Eq:DefEAlpha} and \eqref{Eq:DefAB} into account one can write
\begin{align*}
  \mu_r &= \sum_{s_1 \DT, s_r = 1}^m w_{i_{s_1}j_{s_1}k_{s_1}l_{s_1}}
  \dots w_{i_{s_r}j_{s_r}k_{s_r}l_{s_r}} \sum_{\substack{J \subseteq \{1
      \DT, n\}\\ \#J = r}} \sum_{I \subseteq \{1 \DT, r\}} \det
  (A_{\E}^J)_I \det (B_{\E}^J)_{\bar I} \\
  &= \sum_{s_1 \DT, s_r = 1}^m w_{i_{s_1}j_{s_1}k_{s_1}l_{s_1}} \dots
  w_{i_{s_r}j_{s_r}k_{s_r}l_{s_r}} \sum_{\substack{J \subseteq \{1
      \DT, n\}\\ \#J = r}} \sdet(A_{\E}^J,B_{\E}^J).
\end{align*}
Theorem \ref{Th:Main} is proved.
\end{proof}

\subsection{Final remarks and further research}

\subsubsection{Geometry of $4$-graphs} For every $1 \le i<j<k<l \le n$ consider
two different tetrahedra with the vertices $i,j,k,l$, and call them
$4$-edges $T_1$ and $T_2$. A $4$-graph is defined a union of several
$4$-edges glued by vertices.

The main result of the paper, Theorem \ref{Th:Main}, expresses a
coefficient $\mu_r$ of the characteristic polynomial as a homogeneous
(degree $r$) polynomial of the coefficients $w_{ijkl}$. If one puts a
coefficient $w_{ijkl}$ on the $4$-edge $T_1$ and $w_{iklj}$, on the
$4$-edge $T_2$, then this polynomial becomes a sum over the set of all
$4$-graphs with $r$ edges. The summand corresponding to a graph $\E$
is the product of weights of all the edges of $\E$ times an integer
coefficient $c_{\E}$ described in the theorem (sum of shuffle
determinants of minors of the matrices $A_{\E}$ and $B_{\E}$).

Matrix-tree theorems \ref{Th:MTT} and \ref{Th:PfT} have similar
structure with ordinary graphs and $3$-graphs in place of the
$4$-graphs. In Theorem \ref{Th:MTT} the coefficient $c_{\E}$ is equal
to $n$ if $\E$ is a tree and is zero otherwise. In Theorem
\ref{Th:PfT} one has $\mu_{n-1} = \bigl(\left.\name{Pf}
y\right|_V\bigr)^2$, so $c_{\E} = n^2 \sum \delta(G_1) \delta(G_2)$
where the sum is taken over all representations $\E = G_1 \sqcup G_2$
of $\E$ as a union of two $3$-trees. The formula for $c_{\E}$ in
Theorem \ref{Th:Main} is explicit, but unlike theorems \ref{Th:MTT}
and \ref{Th:PfT} it is not related to the geometry of the underlying
$4$-graph. Finding such a relation would be an interesting
combinatorial problem to solve.

\subsubsection{Structure of $\Lie_n$ as a Lie algebra and as a
  representation of $S_n$} For any group $G$ and its representation
$V$ the elements $x \in \mathcal L(V) \subset k[G]$ act in the
representation $V$. This action may have a kernel; denote it $K(V)
\subset L(V)$ (and $K_n \subset L_n$ if $V$ is the permutation
representation of $S_n$).

\begin{conjecture}
$\dim \mathcal L_n/K_n = (n-1)!$. The repeated commutators of
  Kirchhoff differences
 \begin{equation*}
[\dots[\kappa_{1i_1},\kappa_{2i_2}],\kappa_{3i_3}],\dots],\kappa_{n-1,i_{n-1}}]
 \end{equation*}
for all $i_1 \DT, i_{n-1}$ such that $s+1 \le i_s \le n$ for all $s = 1
\DT, n-1$ form a basis in $\mathcal L_n/K_n$.
 \end{conjecture}

We tested the conjecture numerically for small $n$; yet it is not
proved at the moment.

For any $n$ consider the embedding $\iota_n: S_n \to S_{n+1}$ of $S_n$
to $S_{n+1}$ as a stabilizer of $(n+1)$; extend it by linearity to the
algebra homomorphism $\iota_n: \Complex[S_n] \to \Complex[S_{n+1}]$.

 \begin{proposition}\label{Pp:Induct}
$\iota_n(\mathcal L_n) \subset \mathcal L_{n+1}$.
 \end{proposition}

 \begin{proof}
Let $u = \sum_{\sigma \in S_n} a_\sigma \sigma \in \mathcal L_n$;
consider the action of $\Grp_m\lr{\iota_n(u)}$ and
$\Alg_m\lr{\iota_n(u)}$ on $x \bydef x_{i_1} \wedge \dots \wedge
x_{i_m}$ where $1 \le i_1 \DT< i_m \le n+1$. If $i_m \le n$ then
$\Grp_m\lr{\iota_n(u)}(x) = \Grp_m\lr{u}(x) = \Alg_m\lr{u}(x) =
\Alg_m\lr{\iota_n(u)}(x)$, hence $\iota_n(u) \in \mathcal L_{n+1}$.

Let now $i_m = n+1$. Then $\Grp_m\lr{\iota_n(u)}(x) = 
\Grp_{m-1}\lr{u}(x_{i_1} \DT\wedge x_{i_{m-1}}) \wedge x_{n+1}$. On the 
other hand,
 \begin{equation*}
\Alg_m\lr{\iota_n(u)}(x) = \left(\sum_{\sigma \in S_n} a_\sigma \sum_{p=1}^n 
x_{i_1} \DT\wedge x_{\sigma(i_p)} \DT\wedge 
x_{i_{m-1}}\right) \wedge x_{n+1} + \sum_{\sigma \in S_n} a_\sigma \cdot x.
 \end{equation*}
By Remark \ref{Rm:Sum0} the last term in the equation above is
zero. Thus,
 \begin{equation*}
\Alg_m\lr{\iota_n(u)}(x) = \Alg_{m-1}\lr{u}(x_{i_1} \DT\wedge x_{i_{m-1}}) 
\wedge x_{n+1},
 \end{equation*}
and therefore $\Grp_m\lr{\iota_n(u)}(x) = \Alg_m\lr{\iota_n(u)}(x)$,
which means $\iota_n(u) \in \mathcal L_{n+1}$.
 \end{proof}

Proposition \ref{Pp:Induct} allows to consider the inductive limit
$\Lie_\infty$ of $\Lie_2 \stackrel{\iota_2}{\hookrightarrow} \Lie_3
\stackrel{\iota_3}{\hookrightarrow} \dots$. It is a representation of
the group $S_\infty$ of finitely supported permutations of $\{1, 2,
\dots\}$ and a Lie subalgebra of $\Complex[S_\infty]$ (conjecturally,
generated by the Kirchhoff differences $\kappa_{ij} = 1-(ij)$, $1 \le i <
j$). Very few is known yet about both structures on $\Lie_\infty$.

\subsubsection{Lie elements and embedded graphs} Let $(i_1 j_1) \DT,
(i_m j_m)$ be a sequence of transpositions in $S_n$ or, which is the
same, the numbered edges of a graph $\Gamma$ with the vertices $1 \DT,
n$. There exists (see \cite{BurmanZvonkine,BurmanFesler}) a uniquely
defined embedding of $\Gamma$ into a sphere $M$ with handles and holes
sending the vertices of $\Gamma$ to the boundary of $M$; distribution
of the vertices among components of the boundary coincides with the
cyclic structure of the permutation $\sigma = (i_1 j_1) \dots (i_m
j_m)$. Using this construction and the Lie element property of
$1-(ij)$ it is possible, in particular, to obtain a formula for the
number of ``minimal'' (one-faced) embeddings of any graph; see
\cite[Theorem 2]{BurmanZvonkine}. Probably, one can build a version of
this theory for $3$-graphs, $4$-graphs etc.; Lie element properties of
$\nu_{ijk}$ and $\eta_{ijkl}$ will give some important information
about the embeddings.

\end{document}